\documentclass[preprint]{elsarticle}
\usepackage{graphicx} 
\usepackage{amsmath}
\usepackage{tikz}
\usepackage{tikz-3dplot}
\usepackage{amsmath, amsthm, amscd, amsfonts, amssymb,color, tikz, enumerate}
\usepackage{amsfonts}
\usepackage{amsmath,amsthm,amscd,amssymb,mathrsfs,setspace}
\usepackage{latexsym,epsf,epsfig}
\usepackage{bm}
\usepackage{color}
\usepackage[hmargin=2.5cm,vmargin=3cm]{geometry}
\usepackage{mathtools}
\usepackage{dirtytalk}
\usepackage{cancel}
\usepackage[shortlabels]{enumitem}
\usepackage [english]{babel}
\usepackage [autostyle, english = american]{csquotes}
\usepackage{mathtools}
\usepackage{multicol}
\MakeOuterQuote{"}

\newtheorem{theorem}{Theorem}[section]
\newtheorem{lemma}[theorem]{Lemma}
\newtheorem{proposition}[theorem]{Proposition}

\theoremstyle{remark}
\newtheorem{remark}{Remark}[section]
\numberwithin{equation}{section}

\theoremstyle{plain}

\newcommand{\comments}[1]{}

\newcommand{\N}{\mathbb N}
\DeclareMathOperator*{\esssup}{ess\,sup}

\newcommand{\be}{\begin{equation}}
\newcommand{\ee}{\end{equation}}
\newcommand{\bes}{\begin{equation*}}
\newcommand{\ees}{\end{equation*}}

\newcommand{\bea}{\begin{eqnarray}}
\newcommand{\eea}{\end{eqnarray}}
\newcommand{\beas}{\begin{eqnarray*}}
\newcommand{\eeas}{\end{eqnarray*}}

\newcommand{\red}{\color{red}}

\def\<{\langle} 
\def\>{\rangle}

\title{Reformulation and Interpretation of the Regularity Criterion for 3D NSE \\ Based on Finitely Many Observations}
 \author{\small \begin{tabular}[t]{c@{\extracolsep{2em}}cc@{\extracolsep{2em}}}
         Abhishek Balakrishna & Animikh Biswas \\ 
 \it USC ~~&~~\it UMBC \\ 
ab45315@usc.edu & abiswas@umbc.edu 
\end{tabular}}
\date{March 2023}

\usepackage{subcaption} 

\newcommand{\MyElev}{70}   
\newcommand{\MyAzim}{120}   

\newcommand{\PanelScale}{1.2}

\newcommand{\CubeCoords}{%
  \coordinate (O)   at (0,0,0);
  \coordinate (A)   at (1,0,0);
  \coordinate (B)   at (0,1,0);
  \coordinate (C)   at (0,0,1);
  \coordinate (AB)  at (1,1,0);
  \coordinate (AC)  at (1,0,1);
  \coordinate (BC)  at (0,1,1);
  \coordinate (ABC) at (1,1,1);
}
\newcommand{\DrawCube}{%
  \foreach \u/\v in {O/A,O/B,O/C, A/AB,A/AC, B/AB,B/BC, C/AC,C/BC, AB/ABC,AC/ABC,BC/ABC}
    \draw[gray!70] (\u) -- (\v);
}
\newcommand{\DrawGuides}{%
  \draw[gray!70,dashed] (O) -- (ABC);
  \draw[gray!70,dashed] (A) -- (B);
  \draw[gray!70,dashed] (A) -- (C);
  \draw[gray!70,dashed] (B) -- (C);
}
\newcommand{\HighlightTetra}[5]{%
  \filldraw[fill=#1!30,draw=#1,opacity=.35,very thick] (#2)--(#3)--(#4)--cycle;
  \filldraw[fill=#1!25,draw=#1,opacity=.35,very thick] (#2)--(#3)--(#5)--cycle;
  \filldraw[fill=#1!20,draw=#1,opacity=.35,very thick] (#2)--(#4)--(#5)--cycle;
  \filldraw[fill=#1!15,draw=#1,opacity=.35,very thick] (#3)--(#4)--(#5)--cycle;
}

\begin{document}
\begin{abstract}
   We revisit and sharpen a recent ``observable'' regularity criterion for the three-dimensional Navier--Stokes equations on the periodic cube by requiring only finitely many measurements of the flow on a time interval $[0,T]$.. Two data models are treated: (i) \emph{modal} observations (a finite set of low Fourier modes), and (ii) \emph{nodal} observations, i.e. values of the velocity field sampled at finitely many points on a uniform grid. The key upgrade is a piecewise \emph{linear} interpolant built on a fixed five-tetrahedra subdivision of each grid cube, which removes the mollification step used previously and yields an explicit $H^{1}$ control of the interpolant purely in terms of the measured data. The criterion is also shown to be both necessary and sufficient for regularity.
\end{abstract}

\maketitle

\section{Introduction}
Let $\Omega\subset \mathbb{R}^3$ be a bounded domain with piece-wise smooth boundary $\Gamma$. Assume $\Omega$ is filled with a viscous, incompressible fluid modeled by the 3D Navier-Stokes equations (3D NSE):
 \begin{equation}\label{1}
\begin{split}
    \frac{\partial u}{\partial t}+(u\cdot\nabla)u-\nu\Delta u+\nabla p=f;~~~~~\nabla\cdot u=0;~~~~~
    u|_{t=0}=u_0.
\end{split}
\end{equation}
 For simplicity,  we consider our domain $\Omega\equiv[0,L]^3$ with periodic boundary conditions in the space variables.
 Here, $u$ is the fluid velocity field, $p$ is the associated pressure, and $f\equiv f(x)$ is the given body force. 

 We utilize standard notions of (Leray-Hopf) weak solutions and strong solutions for 3D NSE; see for instance \cite{Leray}. It suffices to say that {\em weak solutions} to \eqref{1} are taken in the sense of equation 4.3 \cite{fmrt}. We say that a weak solution $u$ is {\em regular} over the time interval $[0,T]$ when $\|u(t)\|_{H^1}$ is bounded for $t\in [0,T]$\footnote{Here and below $W^{s,q}(\Omega)$ is the standard Sobolev space and $H^s(\Omega)=W^{s,2}(\Omega)$}. In the latter case, weak solutions are strong in the sense that they satisfy the 3D NSE in a point-wise sense. The global-in-time existence of strong solutions to the 3D NSE is a well known open problem. In the absence of Hadamard well-posedness for weak or strong solutions, several so called regularity criteria have been developed. These criteria constitute additional conditions on weak solutions which ensure they are, in fact, regular. We now remark on several established examples. 
 
 Among the earliest results is the classical Prodi-Serrin-Ladyzhenskaya regularity criterion \cite{Serrin}, which states that a weak solution is regular on $(0,T)$ provided $u \in L^s(0,T;L^q(\Omega))$ for exponents satisfying the scaling condition $\frac{3}{q}+\frac{2}{s}\le 1$ with $2<s<\infty$ and $3<q< \infty$. The borderline case $u\in L^\infty(0,T;L^3(\Omega))$ was later established by Seregin and Sverak \cite{sverak}. 
Another celebrated regularity criteria, applicable also to the inviscid Euler equations,  is  due to Beale-Kato-Majda \cite{BKM}  and it assets that if the vorticity $\omega = \nabla \times u \in L^1(0,T; L^\infty(\Omega))$, then $u$ is regular. Other criteria, similar to the Prodi-Serrin-Ladyzhenskaya, but  involving either the velocity gradient $\nabla u$, or the pressure $p$, or only one component of the velocity,  can be found in \cite{ber, chaelee, daveiga, kukavica, zhou, he, kukaz, pokorny, titi, zhou02}

Each of the regularity criteria mentioned above requires the knowledge of (an adequate norm of) the solution  in space. {\em The principal novelty of the regularity criterion proposed by the authors in \cite{BB,BB*}  is that it is based on observing/measuring the solution for only a finite collection of data}. By data, we  mean a finite collection of functionals evaluated on a trajectory, such as nodal observations, i.e., point values of the velocity field at finitely many points on an uniform grid or a  finite set of Fourier modes---see the related notion of determining functionals \cite{chueshov}. The proof of the main result relies on {\em nudging}, a  {\em data assimilation}  technique first introduced in the context of fluid dynamics by Azouani-Olson-Titi  \cite{AOT}. This involves constructing an approximation  function $w$ using the observed data  from the reference solution $u$ in \eqref{1}. The $H^1$- norm of $w$ is then used to formulate the regularity criterion. The approximation $w$ is obtained by constructing a piecewise constant function and then mollifying it to be able to obtain it's $H^1$- norm. Though this was sufficient to obtain the regularity criterion, the mollification process makes it difficult to establish a direct relationship between the observed data and the regularity criterion.

In this paper we modify the interpolation function to reformulate our regularity criterion in \cite{BB*}. We consider a piecewise linear interpolation function which allows us to skip the mollification step and directly obtain the $H^1$- norm of the interpolation function $w$. This allows us to obtain a regularity criterion on $u$ as an explicit condition on the observations of $u$. This will provide a deeper insight into the criteria and the role played by observed data in regularity. Additionally, we show that $H^1$- norm of the interpolation operator can in fact be written in terms of a very natural approximation of its directional derivatives.  We will also establish that this reformulated regularity criterion is both necessary and sufficient for regularity.

\section{Notation and Preliminaries}
We adopt the convention that $\mathbf L^2(\Omega) \equiv (L^2(\Omega))^3$, with the analogous meaning for $\mathbf H^{\alpha}(\Omega)$, the standard $\mathbf L^2(\Omega)$-based Sobolev space of order $\alpha \in \mathbb R$. We denote by $(\cdot,\cdot)$ the inner-product  and $|\cdot|$ the norm in $\mathbf L^2(\Omega)$; similarly, in $H^1(\Omega)$, by $((\cdot,\cdot))$ and $||\cdot||$ respectively. We define $\mathcal{V}$ as the space of $L$-periodic trigonometric polynomials from $\mathbb{R}^3$ to $\mathbb{R}^3$ which are divergence free and zero mean. $H$ denotes the closure of $\mathcal{V}$ in $\mathbf L^2(\Omega)$ while $V$ denotes the closure of $\mathcal{V}$ in $\mathbf H^1(\Omega)$, as is standard \cite{fmrt}. Additionally, we denote by $P_{\sigma}: \mathbf L^2(\Omega)\to H$ the Leray-Hopf  projector.

 We now cast \eqref{1} in terms of the appropriate operators.
	Let $D(A)=\mathbf{H}^2(\Omega)\cap V$ and  consider $A : H \supset D(A)\to H$ as the unbounded operator associated to the bilinear form, $Au(v)=((u,v))$.
	We recall that $A$ is positive, self-adjoint and has a compact inverse; hence, there is a complete orthonormal set of
	eigenfunctions $\phi_{j}\in H$, such that $A\phi_{j} = \lambda_{j}\phi_{j}$ and we may order $0<\lambda_{1} \leq \lambda_{2} \leq \lambda_{3} \leq \dots $  the eigenvalues (repeated according to multiplicity). We also have the Poincar\'e inequality: ~$\lambda_1^{1/2}|v|\leq\|v\|,~~ v\in V.$
	Now we define the bilinear form $B:V\times V\to V'$ by
	\begin{equation}\label{orthoganal}
	\langle B(u,v),w\rangle_{V'\times V}=\big(((u\cdot\nabla) v),w\big), ~\text{ which satisfies }~B(u,w,w) = 0,~~\forall~ u\in V,~w\in V.
	\end{equation}
	\noindent
With the above notation, by applying $P_{\sigma}$ to \eqref{1}, we may express the 3D NSE and the relevant data assimilation algorithm, given by the solution $w$, in the following functional form:

\begin{align}\label{3dnav}
\left\{
\begin{aligned}
\frac{\partial u}{\partial t}+B(u,u)+\nu Au &= f,\\
\nabla\cdot u &= 0,\\
u(t=0) &= u_0 .
\end{aligned}
\right.
\tag{2.2}
\qquad
\left\{
\begin{aligned}
\frac{\partial w}{\partial t}+B(w,w)+\nu Aw &= f + \mu(Iu-Iw),\\
\nabla\cdot w &= 0,\\
w(t=0) &= 0 .
\end{aligned}
\right.
\end{align}

\vskip.2cm

Here, $I$ is an adequate {\em interpolation operator}, which we elaborate  in the following section
\section{Interpolation Operator}
As part of the data assimilation algorithm, we will incorporate observed data to our model through interpolation operators. In this work, we primarily consider Fourier modes (modal observation) and nodal observations of a reference solution $u$ of \eqref{1}. The data-assimilated solution will correspond to a related system of comparable type. We now introduce our interpolation operator.

For the modal case, $I_h u=P_N(u)$ with $h\sim 1/\lambda_N^{1/2}$, where $P_K$ denotes the orthogonal projection onto the space spanned by the first $N$ eigenvectors of the Stokes operator $A$. The modal interpolation function satisfies
\begin{equation}\label{intest}
	|I_h(v)|\leq c|v|~~ \forall v\in L^2(\Omega) \text{ and } |I_h(v)-v| \leq ch\|v\|~~ \forall v\in H^1(\Omega).
\end{equation}

In the nodal case, the domain $\Omega$ is partitioned into tetrahedrons and the reference solution is observed on the vertices of the tetrahedron. In this paper we consider a specific partition obtained by first dividing the domain into $N$ smaller cubes, $Q_{\alpha}$, of side length $h$, where $\alpha\in\mathcal{J}=\left\{(j, k, l)\in \mathbb{N}\times \mathbb{N}\times \mathbb{N}~:~1\leq j, k, l\leq \N^{1/3}\right\}$.
{\red What is $N^{1/3}$} Each cube $Q_{\alpha}$ is further divided into five tetrahedrons $T_1^\alpha,\dots,T_5^\alpha$ as shown in the figure below. Note that for each $i\in\{1,\dots,5\}$, $T_i^{\alpha_1}$ can be obtained by translating $T_i^{\alpha_2}$ for all $\alpha_1,\alpha_2\in\mathcal{J}$. The reference solution is observed on the nodes of the tetrahedrons (which are the same as that of the cubes). The observations are then used to construct a piece-wise linear function $\mathcal{I}u$ on the entire domain $\Omega$ with $\mathcal{I}u$ being linear in each tetrahedron. $\mathcal{I}u$ is then corrected to have zero average to give $Iu$.  Mathematically,  $\mathcal{I}u$ is obtained as follows:
\begin{equation}\label{iu}
    \mathcal{I}u(x)=\sum_{k=1}^4 a_k^{\alpha,i}u(v_k^{\alpha,i}),\quad \text{for}\quad x=\sum_{k=1}^4a_k^{\alpha,i}v_k^{\alpha,i}\in T_i^\alpha,\quad 0\leq a_k^{\alpha,i}\leq 1,\quad \sum_{k=1}^4a_k^{\alpha,i}=1,
\end{equation}
where $v_j^{\alpha,i}$ are the vertices of the tetrahedron $T_i^\alpha$. Alternately, writing $a_4^{\alpha,i}=1-\sum_{j=1}^3a_j^{\alpha,i}$, we obtain
\begin{equation}\label{reference}
    \mathcal{I}u(x)=u(v_4^{\alpha,i})+\sum_{k=1}^3 a_k^{\alpha,i}(u(v_k^{\alpha,i})-u(v_4^{\alpha,i}))\quad \text{for}\quad x=v_4^{\alpha,i}+\sum_{k=1}^3 a_k^{\alpha,i}(v_k^{\alpha,i}-v_4^{\alpha,i})\in T_i^\alpha.
\end{equation}
$Iu$ is then defined as $Iu(x)=\mathcal{I}u(x)-\frac{1}{L^3}\int_{\Omega}\mathcal{I}u(y)\,dy$.
We see in \eqref{reference} that $v_4^{\alpha,i}$ is the ``reference'' vertex with respect to which any point in $T^\alpha_i$ is represented.
\begin{figure}[htbp]
\centering
\tdplotsetmaincoords{\MyElev}{\MyAzim}

\begin{subfigure}{.32\textwidth}
\centering
\begin{tikzpicture}[tdplot_main_coords,scale=\PanelScale,line cap=round]
  \CubeCoords \DrawCube \DrawGuides
  \HighlightTetra{blue}{O}{A}{B}{C}
\end{tikzpicture}
\subcaption{$T_1^{\alpha}$}
\end{subfigure}\hfill
\begin{subfigure}{.32\textwidth}
\centering
\begin{tikzpicture}[tdplot_main_coords,scale=\PanelScale,line cap=round]
  \CubeCoords \DrawCube \DrawGuides
  \HighlightTetra{orange}{A}{B}{C}{ABC}
\end{tikzpicture}
\subcaption{$T_2^{\alpha}$}
\end{subfigure}\hfill
\begin{subfigure}{.32\textwidth}
\centering
\begin{tikzpicture}[tdplot_main_coords,scale=\PanelScale,line cap=round]
  \CubeCoords \DrawCube \DrawGuides
  \HighlightTetra{green!60!black}{A}{B}{AB}{ABC}
\end{tikzpicture}
\subcaption{$T_3^{\alpha}$}
\end{subfigure}

\vspace{0.6em}

\begin{subfigure}{.32\textwidth}
\centering
\begin{tikzpicture}[tdplot_main_coords,scale=\PanelScale,line cap=round]
  \CubeCoords \DrawCube \DrawGuides
  \HighlightTetra{red}{A}{C}{AC}{ABC}
\end{tikzpicture}
\subcaption{$T_4^{\alpha}$}
\end{subfigure}\hfill
\begin{subfigure}{.32\textwidth}
\centering
\begin{tikzpicture}[tdplot_main_coords,scale=\PanelScale,line cap=round]
  \CubeCoords \DrawCube \DrawGuides
  \HighlightTetra{purple}{B}{C}{BC}{ABC}
\end{tikzpicture}
\subcaption{$T_5^{\alpha}$}
\end{subfigure}

\caption{Five copies of the cube; each panel highlights one tetrahedron of the standard 5-tetrahedra triangulation.}
\label{fig:cube-5tet-panels-same-view-half}
\end{figure}
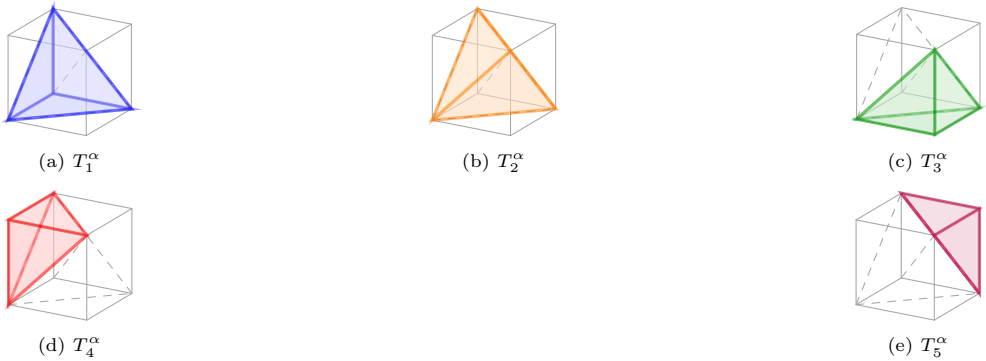

$Iu$ serves as an approximation of $u$. $\mathcal{I}u$ satisfies the following inequality \cite[Theorem~4.4.20]{brenscott}:
\begin{equation}\label{in2}
    |u(t)-\mathcal{I}u(t)|\leq ch^2\|u(t)\|_{H^2(\Omega)}.
\end{equation}
Therefore, using \eqref{in2} and, the fact that $u$ has zero average, and Holder's inequality, we get
\begin{equation}\label{in}
    |u(t)-Iu(t)|\leq |u(t)-\mathcal{I}u(t)|+\left|\frac{1}{L^3}\int_\Omega\mathcal{I}u\right|\leq |u(t)-\mathcal{I}u(t)|+\left|\frac{1}{L^3}\int_\Omega\mathcal{I}u-u\right|\leq ch^2\|u(t)\|_{H^2(\Omega)}.
\end{equation}
For the remainder of this paper, $c$ represents a constant independent of the function in the inequality.

\section{$H^1$-norm of $Iu$ purely in terms of data}
Before we state the reformulated regularity criterion, we will evaluate the $H^1$-norm of $Iu$ and express it explicitly in terms of the observed data.

In the modal case this is quite direct as $P_Nu$ represents the observed data and
\begin{equation}\label{pk}
    \|\nabla Iu\|=|\nabla P_N(u)|^2\sim \sum_{\lambda_k\leq N^2}\lambda_{k}|\hat{u}(k)|^2,~h\sim 1/\lambda_N^{1/2}.
\end{equation}
In the nodal case providing an estimate for $\|Iu\|$ explicitly in terms of the data requires more work, which we present now. Let $e_j$ denote the unit vector in the $x_j$ direction and $w_j^{i}=v_j^{\alpha,i}-v_4^{\alpha,i}$ for $j=1,2,3$. Note that $w_j^{i}$ does not depend on $\alpha$ as  for each $i\in\{1,\dots,5\}$ al tetrahedrons $T^\alpha_i$ are translations of each other and $w_j^{i}=v_j^{\alpha,i}-v_4^{\alpha,i}$ are edge vectors of $T_i^\alpha$.   For each $T^\alpha_i$, the vectors $w_j^{i}$ are linearly independent. Hence, there exists an invertible matrix $M^{i}$ such that
\begin{equation}\label{ej}
    [e_1,e_2,e_3]=[\widehat{w}_1^{i},\widehat{w}_2^{i},\widehat{w}_3^{i}]*M^{i},\quad\text{or equivalently}\quad e_j=\sum_{k=1}^3M^{i}_{k,j}\widehat{w}_{k}^{i}=\sum_{k=1}^3\frac{M^{i}_{k,j}w_{k}^{i}}{|w_{k}^{i}|},
\end{equation}
where $\widehat{w}_k^{i}$ is the unit vector in the direction of $w_k^{i}$. Note that since $M_{k,j}^i$ is defined using unit vectors, it is independent of $h$. Consequently,  for $x\in T^\alpha_i$, using \eqref{iu} and and\eqref{ej} we obtain
\begin{equation}\label{ij}
\begin{split}
    \partial_jIu(x)&=\lim_{s\to 0}\frac{1}{s}[Iu(x+se_j)-Iu(x)]\\
    &=\lim_{s\to 0}\frac{1}{s}\left[Iu\left(v_{4}^{\alpha,i}+\sum_{k=1}^3\left(a_{k}^{\alpha,i}+\frac{sM^i_{k,j}}{|w_{k}^i|}\right)w_{k}^i\right)-Iu\left(v_{4}^{\alpha,i}+\sum_{k=1}^3a_{k}^{\alpha,i}w_{k}^i\right)\right]\\
    &=\lim_{s\to 0}\frac{1}{s}\left[\sum_{k=1}^3 \frac{sM^i_{k,j}}{|w_{k}^i|}\left(u(v_k^{\alpha,i})-u(v_4^{\alpha,i})\right)\right]=\sum_{k=1}^3 \frac{M^i_{k,j}}{|v_{k}^{\alpha,i}-v_{4}^{\alpha,i}|}\left(u(v_k^{\alpha,i})-u(v_4^{\alpha,i})\right)
    \end{split}
\end{equation}
Let $D^{\alpha,i}(u)$ be the matrix define as 
\begin{equation}\label{d}
    D^{\alpha,i}_{l,k}(u)=\frac{u_l(v_k^{\alpha,i})-u_l(v_4^{\alpha,i})}{|v_{k}^{\alpha,i}-v_{4}^{\alpha,i}|}.
\end{equation}
Then, we may rewrite \eqref{ij} as $\nabla Iu(x)= D^{\alpha,i}(u)M$ for $x\in T_i^\alpha$.
Therefore
\begin{equation}\label{D}
    \|M^{-1}\|_{L^\infty}^{-2}\sum_{i,\alpha}\|D^{\alpha,i}(u,t)\|^2_{L^2}\leq\|\nabla Iu(t)\|^2_{L^2}\leq  \|M\|^2_{L^\infty}\sum_{i,\alpha}\|D^{\alpha,i}(u,t)\|^2_{L^2}.
\end{equation}
We therefore have an equivalence between $\|\nabla Iu(t)\|_{L^2}$ and $\|D^{\alpha,i}(u,t)\|_{L^2}$, where $D^{\alpha,i}(u,t)$ can be regarded as coarse scale directional derivatives along the edges of the tetrahedrons. This is an important novelty of this paper. In previous applications of data assimilation in the context of 3D NSE, the interpretability of $\nabla Iu$ was not clear as $Iu$ was defined as a piecewise constant function which was then mollified. We have shown that in the case of a linear interpolation operator, $\nabla Iu$ can be interpreted as an approximation of the first derivative of $u$.

\section{Main Result and Brief Discussion}
 We now state our regularity criterion in the modal and nodal cases. Before doing so, let us introduce the quantity $M_h$ which relates the relevant physical quantities $\nu$, $\lambda_1$, $T$, and  $h$ (which is tied to $\alpha$, as above). It serves as an upper bound for the $\|Iu(t)\|$ for almost all $t\in(0,T)$. Using \eqref{pk} and the equivalence in \eqref{D} we may now write $M_{h,[0,T]}$ explicitly in terms of the data as follows:
\begin{equation}\label{mh}
    \begin{split}
       M_{h,[t_0,T]}^2=\esssup_{t_0\leq t\leq T}
       \begin{dcases}
        \displaystyle\|\nabla P_N(u)\|_{L^2(\Omega)}^2\sim \sum_{\lambda_k\leq N^2}\lambda_{k}|\hat{u}(k)|^2,~h\sim 1/\lambda_N^{1/2}&(\text{Modal})\\[10pt]
         \displaystyle  \displaystyle \sum_{\alpha}\sum_{i=1}^5\|D^{\alpha,i}(u,t)\|^2_{L^2}.
           &(Nodal)
        \end{dcases}
    \end{split}
\end{equation}
We now state our reformulated regularity criterion in the following theorem.
  \begin{theorem}\label{criteria}
           Let $u$ be a Leray-Hopf weak solution to the 3D NSE given by the equations on the left of \eqref{3dnav} such that $u(0)\in V$ and $M_{h,[0,T]}$ be defined by \eqref{mh}.
           If there exists an $h>0$ such that
	    \begin{equation}\label{rere}
	       \max\left\{\nu\lambda_1,\frac{W_{h,[0,T]}^4}{\nu^3},\frac{c\|\nabla u(0)\|_{L^2(\Omega)}^4}{\nu^3}\right\}\leq\frac{\nu}{ch^2}, \quad\text{ where }\quad  W_{h,[0,T]}^2=\frac{c}{\nu^2\lambda_1} \|f\|_{L^2(\Omega)}^2 + cM_{h,[0,T]}^2,
	    \end{equation}
	    for some adequate, non-dimensional absolute constant $c$, then $u$ is regular, with $\|u(t)\|\lesssim W_h$ for $t\in [0,T]$.
        \end{theorem}
The proof would be along the lines of \cite[Theorem~1.1]{BB*} as $Iu$ satisfies \eqref{in}. We may in fact improve the above statement by working with less data. Since $u(0)\in V$, by \cite[Theorem 7.2]{fmrt} $u(t)\in V$ for $t\in T_*$ for some $T_*>0$. We hence only need to check for regularity on the interval $(t_0,T]$ for some $t_0\in (0,T_*)$, allowing us to work with data on the interval $[t_0,T]$ instead of $[0,T]$. Additionally, the regularity criterion obtained in \eqref{rere} is both necessary and sufficient for regularity. These facts are presented in the following theorem:
 \begin{theorem}\label{suff}
          Let $u$ be a weak solution to the 3D NSE given by \eqref{3dnav} such that $u(0)\in V$. Let $T_*$ be as in \cite[Theorem 7.2]{fmrt}, $t_0\in(0,T_*)$ and $M_{h,[t_0,T]}$ be defined by \eqref{mh}. Then $u$ is regular if and only if there exists an $h>0$ such that
	    \begin{equation}\label{regcrit2}
	       \max\left\{\nu\lambda_1,\frac{W_{h,[t_0,T]}^4}{\nu^3},\frac{c\|u(t_0)\|^4}{\nu^3}\right\}\leq\frac{\nu}{4ch^2}, \quad\text{ where }\quad  \tilde{W}_{h}^2=\frac{c|f|^2}{\nu^2\lambda_1}  + cM_{h,[t_0,T]}^2.
	    \end{equation}
    \end{theorem}
\begin{proof}
    If $u$ satisfies the criterion in \eqref{regcrit2}, then by Theorem~\ref{criteria}, $u$ is regular on the time interval $[t_0,T]$. Since $u$ is regular on $[0,t_0]$ by definition of $t_0$, we see that $u$ is regular on $[0,T]$.

    Now, let $u$ be regular on the time interval $[0,T]$. To show $u$ satisfies \eqref{regcrit2}, we need to prove $M_{h,[t_0,T]}$ is bounded as $h\to 0$. This is because by the equality on the right of \eqref{regcrit2}, if $M_{h,[t_0,T]}$ is bounded as $h\to 0$ then so is $W_{h,[t_0,T]}$, which would then allow us to satisfy the first inequality in\eqref{regcrit2} by choosing $h$ sufficiently small.

From \cite{BB*}, we have that $\esssup_{0\leq t\leq T}\|\nabla Iu\|_{L^2}$ remains bounded as $h\to 0.$ Now since $\|\nabla Iu\|^2_{L^2}\sim \sum_{i,\alpha}\|D^{\alpha,i}(u,t)\|^2_{L^2}$ from \eqref{D}, we immediately obtain the statement of this theorem.

\end{proof}

        We conclude with two remarks on the interpretation of the result. First, the regularity criterion obtained here depends only on a finite number of observations of the flow, expressed through the quantities in \eqref{rere}. In particular, the criterion can be verified directly from experimentally accessible data, such as finitely many Fourier modes or nodal measurements of the velocity field. Second, although the inequality in \eqref{rere} involves both the viscosity $\nu$ and the observation scale $h$, Theorem~\ref{suff} shows that the existence of such an $h$ is in fact equivalent to the regularity of the solution. In this sense, the condition provides an observable characterization of regularity for weak solutions of the three–dimensional Navier–Stokes equations. The reformulation presented here clarifies the role of the measured data by expressing the relevant $H^1$ quantity entirely in terms of coarse-scale directional derivatives determined from the observations, thereby eliminating the mollification step required in earlier formulations and providing a more transparent link between measurements and regularity.

\bibliographystyle{elsarticle-num}

\end{document}